\documentclass{article}

\date{}
\usepackage{amsmath}
\usepackage{amssymb}
\usepackage{amsfonts}

\title{\Large\bf Uniform boundedness in function spaces}

\author{{\bf L'ubica Hol\'a}\\
Mathematical Institute, Slovak Academy of Sciences,\\
\v Stef\'anikova 49, SK-814 73 Bratislava, Slovakia\\
{\sf lubica.hola@mat.savba.sk}\\\\
{\bf Ljubi\v sa D.R. Ko\v cinac}\\
University of Ni\v s, Faculty of Sciences and Mathematics, \\18000 Ni\v s, Serbia\\
{\sf lkocinac@gmail.com} }

\newtheorem{theorem}{{\bf Theorem}}[section]

\newtheorem{corollary}[theorem]{{\bf Corollary}}

\newtheorem{definition}[theorem]{{\bf Definition}}
\newtheorem{proposition}[theorem]{{\bf Proposition}}

\newcommand{\naturals}{{\mathbb N}}

\begin{document}

\maketitle

%\begin{abstract}
%We investigate several boundedness properties of function spaces
%considered as uniform spaces.
%\end{abstract}

%\medskip
%\noindent {\sf 2010 Mathematics Subject Classification}: Primary:
%54C35; Secondary: 46A04, 54D20, 54E15.

%\noindent {\sf Keywords}: Function spaces, $\omega$-bounded,
%(strictly) {\sf M}-bounded, (strictly) {\sf H}-bounded, {\sf
%R}-bounded, pseudocompact, Fr\'echet space.

\medskip
\section{Introduction}

We begin with some basic information about uniform and function
spaces. Our topological notation and terminology are standard (see
\cite{engelking}). By $\mathbb N$ and $\mathbb R$ we denote the set
of natural and real numbers, respectively.

\medskip
\noindent {\bf 1.1. Uniform spaces}

\smallskip
Let $X$ be a nonempty set. A family $\mathbb U$ of subsets of
$X\times X$ satisfying conditions
\begin{itemize}
\item[(U1)] each $U\in \mathbb U$ contains the diagonal $\Delta_X=
\{(x,x):x\in X\}$ of $X$;
\item[(U2)] if $U,V\in \mathbb U$, then $U\cap V \in \mathbb U$;
\item[(U3)] if $U\in \mathbb U$ and $V\supset U$, then $V\in \mathbb
U$;
\item[(U4)] for each $U\in \mathbb U$ there is
$V\in \mathbb U$ with $V\circ V := \{(x,y)\in X\times X:\exists z\in
V \mbox{ such that } (x,z)\in V, (z,y) \in V\} \subset U$;
\item[(U5)] for each $U\in \mathbb U$, $U^{-1} :=\{(x,y)\in X\times
X: (y,x)\in U\} \in \mathbb U$
\end{itemize}
is called a \emph{uniformity} on $X$.

Elements of the uniformity $\mathbb U$ are called \emph{entourages}.
For any entourage $U \in \mathbb U$, a point $x\in X$ and a subset
$A$ of $X$ one defines the set
\[
U[x] :=\{y\in X: (x,y)\in U\}
\]
called the \emph{$U$-ball with the center $x$}, and the set
\[
U[A] := \bigcup_{a\in A}U[a]
\]
called the \emph{$U$-neighborhood} of $A$.

\medskip
In \cite{koc1}, several boundedness properties of uniform spaces
were introduced and studied. We recall definitions of those
properties.

\begin{definition}\rm  A uniform space $(X,\mathbb U)$
is called:
\begin{itemize}
\item[(1)] \emph{totally bounded} (resp. \emph{$\omega$-bounded})
if for each $U\in \mathbb U$ there is a finite (resp. countable) set
$A \subset X$ such that $X = U[A]$. $X$ is \emph{$\sigma$-totally
bounded} if it is a union of countably many totally bounded
subspaces;

\item[(2)] \emph{Menger bounded} (or \emph{{\sf M}-bounded} for
short) if for each sequence $(U_n:n\in \mathbb N)$ of entourages
there is a sequence $(F_n:n\in\mathbb N)$ of finite subsets of $X$
such that $X = \bigcup_{n\in\mathbb N}U_n[F_n]$ \cite{koc1, koc2};

\item[(3)] \emph{Hurewicz bounded} (or \emph{{\sf H}-bounded}) if for
each sequence $(U_n:n\in \mathbb N)$ of entourages there is a
sequence $(F_n:n\in\mathbb N)$ of finite subsets of $X$ such that
each $x \in X$ belongs to all but finitely many $U_n[F_n]$
\cite{koc1, koc2};

\item[(4)] \emph{Rothberger bounded} (or \emph{{\sf R}-bounded}) if for
each sequence $(U_n:n\in \mathbb N)$ of entourages there is a
sequence $(x_n:n\in\mathbb N)$ of elements of $X$ such that $X =
\bigcup_{n\in\mathbb N}U_n[x_n]$ \cite{koc1, koc2}.
\end{itemize}
\end{definition}

To each of the above boundedness properties one can correspond a
game on $(X,\mathbb U)$. For example, the game corresponded to {\sf
M}-boundedness is the following. Players ONE and TWO play a round
for each $n\in \mathbb N$. In the $n$-th round ONE chooses an
element $U_n\in\mathbb U$, and TWO responds by choosing a finite set
$A_n \subset X$. TWO wins a play
\[
U_1, A_1; U_2, A_2;  \cdots; U_n, A_n; \cdots
\]
if $X = \bigcup_{n\in\mathbb N}U_n[A_n]$; otherwise ONE wins.

\smallskip
A uniform space $(X,\mathbb U)$ is said to be \emph{strictly {\sf
M}-bounded} if TWO has a winning strategy in the above game
(\cite{koc1, koc2}).

In a similar way we define the games associated to {\sf
H}-boundedness and {\sf R}-boundedness, and  \emph{strictly {\sf
H}-bounded} and \emph{strictly {\sf R}-bounded} uniform space.

\medskip
\noindent {\bf 1.2. Function spaces}

\smallskip
Let $X$ be a Tychonoff space, $(Y,d)$ be a metric space  and
$C(X,Y)$ be the set of continuous functions from $X$ to $Y$. In case
$Y = \mathbb R$ we write $C(X)$ instead of $C(X,\mathbb R)$. If $y
\in (Y,d)$ and $\lambda
> 0$, we put
$S(y, \lambda) = \{z \in Y: d(y,z) < \lambda\}$ and  $B(y, \lambda)
= \{z \in Y : d(y, z) \le \lambda\}$.

There are several uniformities on the set $C(X,Y)$. Let us use the
following notation. $\mathcal F(X)$ is the set of finite subsets of
$X$, $\mathcal K(X)$ the set of compact subsets of $X$, and $C^+(X)$
the set of positive real-valued functions on $X$. If $\varepsilon
> 0$ and $\mathcal C$ is a collection of subsets of $X$, then one
defines on $C(X,Y)$ the uniformity $\mathbb U_{\mathcal C}$ of
uniform convergence on elements of $\mathcal C$ generated by the
sets
\[
W(A,\varepsilon) = \{(f,g) \in C(X,Y)^2: d(f(x),g(x)) < \varepsilon
\ \forall \ x\in A\}, \ \ A \in \mathcal C, \varepsilon > 0.
\]
We call $\mathbb U_{\mathcal F(X)}$ ($\mathbb U_{\mathcal K(X)}$)
the \emph{uniformity of pointwise convergence} (\emph{the uniformity
of uniform convergence on compacta}) and write $\mathbb U_p$ for
$\mathbb U_{\mathcal F(X)}$ and $\mathbb U_k$ for $\mathbb
U_{\mathcal K(X)}$. Topologies on $C(X,Y)$ generated by $\mathbb
U_p$ and $\mathbb U_k$ are the topology $\tau_p$ of pointwise
convergence and compact-open topology $\tau_k$.

Another two uniformities on $C(X,Y)$ we use in this article are the
\emph{uniformity $\mathbb U_u$ of uniform convergence} generated by
the sets of the form
\[
B_{\varepsilon} = \{(f,g)\in C(X,Y)^2: d(f(x),g(x))< \varepsilon \
\forall \ x\in X\}, \ \ \varepsilon > 0,
\]
and the \emph{$m$-uniformity} $\mathbb U_m$ generated by the sets of
the form
\[
D_{\varepsilon} = \{(f,g)\in C(X,Y)^2: d(f(x),g(x))< \varepsilon(x)
\ \forall \ x\in X\}, \ \ \varepsilon \in C^+(X).
\]
Topologies generated by these two uniformities are the
\emph{topology $\tau_u$ of uniform convergence } and
\emph{$m$-topology} $\tau_m$, respectively.

In \cite{hola-koc} we investigated spaces $C(X)$ endowed with the
mentioned topologies considering those spaces as Hausdorff
topological groups with the pointwise addition. In this paper we
consider boundedness properties of  uniform spaces $C(X,Y)$ equipped
with the above mentioned uniformities.

The reader interested in an investigation of spaces $C(X,Y)$ can
consult the papers \cite{DiMaio-Hola-Holy-McCoy, hola-jindal,
Hola-Mccoy, hola-zsilinsky} and the books \cite{arh, dimaio-hola,
MN}.

\section{Results}

\subsection{Preliminary results}

We begin with the following facts, which are either known (see
\cite{koc1, koc2}) or easy to prove.

\medskip
{\bf Fact 1.} A metric space  $(Z,d)$ is $\omega$-bounded if and
only if it is separable. [If $Z$ is $\omega$-bounded, then for each
$n\in \mathbb N$ there is a countable set $A_n$ such that $Z =
\bigcup_{a\in A_n}S(a,1/n)$. Then $A = \bigcup_{n\in\mathbb N}A_n$
is a countable dense subset of $Z$. Conversely, if $Z$ is separable
with a countable dense set $A\subset Z$, then for any $\varepsilon >
0$ and any $x\in Z$ there is an $a\in A$ such that $d(x,a)<
\varepsilon$, hence $Z = \bigcup_{a\in A}S(a,\varepsilon)$.]

\smallskip
{\bf Fact 2.} If  a uniform space $(Z,\mathbb U)$ is a uniformly
continuous image of an $\omega$-bounded (Menger bounded, Hurewicz
bounded, Rothberger bounded) uniform space $(T,\mathbb V)$, then
$(Z,\mathbb U)$ is also $\omega$-bounded (Menger bounded, Hurewicz
bounded, Rothberger bounded).

\smallskip
{\bf Fact 3.} If a uniform space $(T,\mathbb U_T)$ is a subspace of
an $\omega$-bounded,  (strictly) Menger (Hurewicz, Rothberger)
bounded space $(Y,\mathbb U)$, then $(T,\mathbb U_T)$ is also
$\omega$-bounded, (strictly) Menger (Hurewicz, Rothberger) bounded.

\smallskip
{\bf Fact 4.} If  a uniform space $(Z,\mathbb U)$ is
$\sigma$-totally bounded, then $(Z,\mathbb U)$ is strictly Hurewicz
bounded.

\medskip
The following diagram gives relations among the mentioned properties
({\sf TB}, {\sf MB}, {\sf SMB}, {\sf HB}, {\sf SHB}, {\sf RB}, {\sf
SRB} and {\sf met} is notation for totally bounded, {\sf M}-bounded,
strictly {\sf M}-bounded, {\sf H}-bounded, strictly {\sf H}-bounded,
{\sf R}-bounded, strictly {\sf R}-bounded and metrizable,
respectively).

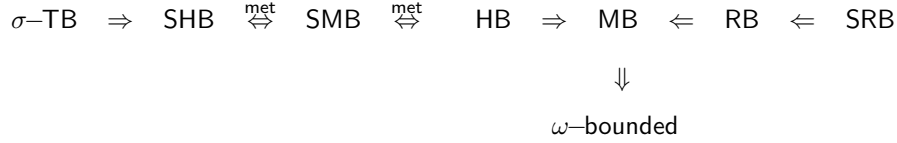
\[
\sigma\!\!-\!\!{\sf TB}\hskip0.3cm  \Rightarrow \hskip0.3cm {\sf
 SHB} \hskip0.3cm \overset{{\sf met}}{\Leftrightarrow } \hskip0.3cm {\sf SMB}
\hskip0.3cm \overset{{\sf met}}{\Leftrightarrow} \hskip0.3cm
\hskip0.3cm {\sf HB} \hskip0.3cm \Rightarrow \hskip0.3cm {\sf MB}
\hskip0.3cm \Leftarrow \hskip0.3cm {\sf RB} \hskip0.3cm \Leftarrow
\hskip0.3cm {\sf SRB}
\]
\[
\hskip4.5cm\Downarrow
\]
\[
\hskip4.3cm\omega\!\!-\!\!\mbox{{\sf bounded}}
\]

\medskip
Evidently,
\[
(C(X,Y),\mathbb U_m) \Rightarrow \ (C(X.Y), \mathbb U_u) \Rightarrow
(C(X,Y),\mathbb U_k) \Rightarrow (C(X,Y),\mathbb U_p)
\]
for each property $\mathcal P \in \{\omega\!\!-\!\!{\rm bounded},
{\sf HB}, {\sf SHB}, {\sf MB}, {\sf SMB},  {\sf RB}, {\sf SRB}\}$.

\medskip
The following two propositions will be used in the next sections.

\begin{proposition}\label{count.base}
Let $(Z,\mathbb U)$ be a uniform space with a countable base. The
following are equivalent:
\begin{itemize}
\item[{\rm (1)}] $(Z,\mathbb U)$ is strictly {\sf H}-bounded;

\item[{\rm (2)}]  $(Z,\mathbb U)$ is strictly {\sf M}-bounded;

\item[{\rm (3)}] $(Z,\mathbb U)$ is {\sf H}-bounded;

\item[{\rm (4)}]  $(Z,\mathbb U)$ is $\sigma$-totally bounded.
\end{itemize}
\end{proposition}
\begin{proof} (1) $\Rightarrow$ (2) and (1) $\Rightarrow$ (3) are
trivial, while (4) $\Rightarrow$ (1) follows from Fact 4. Therefore,
we have to prove (3) $\Rightarrow$ (4) and (2) $\Rightarrow$ (4).

\smallskip
(3) $\Rightarrow$ (4) Let $(U_n:n\in\mathbb N)$ be a sequence of
elements of $\mathbb U$. Without loss of generality one can assume
that each $U_n$ belongs to a countable base $\mathbb B$ of $\mathbb
U$, and that $U_1 \supset U_2 \supset \ldots$. By (3) there is a
sequence $(A_n:n\in\mathbb N)$ of finite subsets of $Z$ such that
each $z\in Z$ belongs to $U_n[A_n]$ for all but finitely many $n$.
For each $n\in \mathbb N$ put $Y_n = \bigcap_{i\ge n}U_i[A_i]$. Then
$Z = \bigcup_{n\in\mathbb N}Y_n$. To finish the proof we should
prove that each $Y_n$ is totally bounded.

Let $n\in \mathbb N$ be fixed and let $U_m\in \mathbb U$. Pick $k\in
\mathbb N$ such that $k
> \max\{m,n\}$. Then $Y_n \subset U_k[A_k] \subset U_m[A_k]$, i.e.
$Y_n$ is totally bounded.

\smallskip
(2) $\Rightarrow $ (4) We suppose that $U_n$s in $\mathbb U$ are as
in the proof of (3) $\Rightarrow$ (4). Let $\sigma$ be the winning
strategy of TWO. For each $n\in \mathbb N$ set $\sigma(U_n) = F_n$,
a finite subset of $Z$, and define $Y_0 = \bigcap_{n\in \mathbb
N}U_n[F_n]$. Denote by $S$ the set of all finite sequences in
$\mathbb N$. For a given $s = (n_1, n_2  \ldots, n_k) \in S$ and
$n\in\mathbb N$, set $F_{n_1, \dots, n_k,n} = \sigma(U_{n_1},
U_{n_2}, \ldots, U_{n_k}, U_n)$, a finite subset of $Y$, define $Y_s
= \bigcap_{n\in\mathbb N} U_{n}[F_{n_1,\dots,n_k,n}]$ and $Y =
\bigcup \{Y_s: s \in S\}$.

\smallskip
\noindent {\bf Claim 1.} $Z = Y$.
\smallskip
Suppose that there is $z \in Z \setminus Y$. Then we can find
inductively $n_1, n_2, \ldots$ in $\mathbb N$   and finite subsets
$F_{n_1}, F_{n_1,n_2}, \ldots $ such that $z \notin
U_{n_1}[F_{n_1}]$, $z\notin U_{n_2}[F_{n_1, n_2}], \ldots$. In this
way we obtain a $\sigma$-play
\[
U_{n_1}, F_{n_1}; U_{n_2}, F_{n_1,n_2}; U_{n_3}, F_{n_1,n_2,n_3};
\ldots
\]
lost by TWO, which is a contradiction.

\smallskip
{\bf Claim 2.} Every $Y_s$ is totally bounded.

\smallskip
Let  $s = (n_1, \ldots, n_k)$ and $U_i \in\mathbb U$ be given.  Take
$n \in \mathbb N$ so that $U_n \subset U_i$. Clearly, $Y_s \subset
U_n[F_{n_1, \ldots, n_k,n}] \subset U_i[F_{n_1, \ldots, n_k,n}]$,
i.e. $Y_s$ is totally bounded. \, \, $\Box$
\end{proof}

\begin{proposition} \label{prop2}  Let $(Z,\mathbb U)$  be a Hausdorff uniform space with a countable base.
The following are equivalent:
\begin{itemize}
\item[{\rm  (1)}] $(Z,\mathbb U)$ is strictly {\sf R}-bounded;

\item[{\rm (2)}] $Z$ is countable.
\end{itemize}
\end{proposition}
\begin{proof} Only (1) implies (2) need the proof. Let $(U_n : n \in
\mathbb N)$ be a countable base for $\mathbb U$ such that $U_1
\supset U_2 \supset \cdots$ and $\bigcap_{n\in\mathbb N}U_n=
\Delta_Z$. The proof is similar to the proof of (2) $\Rightarrow$
(4) in the previous proposition. Let $\varphi$ be a strategy of TWO.
Then for each $U_n$ TWO picks a point $z_n = \varphi(U_n) \in Z$.
Define $Y_0 = \bigcap_{n\in \mathbb N}U_n[z_n]$, and for a given
finite sequence $n_1,\cdots, n_k$ in $\mathbb N$, define
$Y_{n_1,n_2, \ldots,n_k} = \bigcap_{n\in \mathbb N}
U_n[\varphi(U_{n_1}, U_{n_2}, \ldots, U_{n_k}, U_n)]$. As in
Proposition \ref{count.base} define $Y_s$ and $Y$, and prove (with a
suitable modification) that $Z= \bigcup_{s \in S}Y_{s}$.

Finally, using $\bigcap_{n\in\mathbb N}U_n= \Delta_Z$, we easily
prove that each $Y_{s}$ has at most one element, so that $Z$ is
countable. \, \, $\Box$
 \end{proof}

\subsection{$(C(X,Y),\mathbb U_p)$ and $(C(X,Y),\mathbb U_k)$}

The next results give information about the uniform spaces with the
uniformity of pointwise convergence.

\begin{theorem} \label{cp-omega-bound} Let $X$ be a Tychonoff space and $(Y,d)$ be a metric space. The
following are equivalent:
\begin{itemize}
\item[{\rm (1)}] $(C(X,Y),\mathbb U_p)$ is $\omega$-bounded;

\item[{\rm (2)}] $(Y,d)$ is separable.
\end{itemize}
\end{theorem}
\begin{proof} $(1) \Rightarrow (2)$ \ We can consider $Y$ as a subspace of
$(C(X,Y),\mathbb U_p)$. [For every $y \in Y$ consider the mapping
$f_y: X \to Y$ defined by $f_y(x) = y$ for every $x \in X$. The
mapping $\varphi: y \mapsto f_y$ is a homeomorphism of $Y$ onto the
subspace $\varphi(Y)$ of $C(X,Y)$.] By Fact 3, $(Y,d)$ is
$\omega$-bounded, and by Fact 1 $(Y,d)$ must be separable.

\smallskip
$(2) \Rightarrow (1)$  Let $W(A,\varepsilon) \in \mathbb U_p$, where
$A = \{x_1, x_2, \ldots, x_n\} \subset X$ and $\varepsilon  >0$. We
want to find a countable family $\mathcal F \subset C(X,Y)$ such
that
\[
C(X,Y) \subset W(A,\varepsilon)[\mathcal F].
\]
Put $F = \{(f(x_1),f(x_2), \ldots, f(x_n)): f \in C(X,Y)\}$. Then $F
\subset Y^n$. The separability of $(Y,d)$ implies that also $F$ is a
separable subspace of $Y^n$. Let $\{(f_i(x_1),f_i(x_2), \ldots,
f_i(x_n)): i \in \mathbb N\}$ be a countable dense subset of $F$,
and let $\mathcal F= \{f_i: i \in \mathbb N\}$. It is easy to verify
that $C(X,Y) \subset W(A,\varepsilon)[\mathcal F]$. \, \, $\Box$
\end{proof}

\medskip
Recall that a topological space $X$ is said to be \emph{hemicompact}
if there is a sequence of compact subsets of $X$ such that every
compact subset of $X$ is contained in some set of the sequence.

\begin{theorem} \label{p-SHB}  Let $X$ be a pseudocompact space and $(Y,d)$ be a hemicompact
metric space. Then $(C(X,Y),\mathbb U_p)$ is  strictly Hurewicz
bounded.
\end{theorem}
\begin{proof} Let $\{K_n: n \in \mathbb N\}$ be a sequence of compact sets
which is cofinal (with respect to the set inclusion) in the family
of compact sets in $Y$. For every $f \in C(X,Y)$, $f(X)$ is a
compact set in $Y$, thus there is $n \in \mathbb N$ such that $f(X)
\subset K_n$. Thus $(C(X,Y),\tau_p) \subset \bigcup_{n \in \mathbb
N} K_n^X$ is a subspace of a $\sigma$-compact space, i.e.
$(C(X,Y),\mathbb U_p)$ is strictly Hurewicz bounded. \, \, $\Box$
\end{proof}

\begin{corollary} Let $X$ be a pseudocompact space in which every compact set is
finite, and $(Y,d)$ be a hemicompact metric space. Then
$(C(X,Y),\mathbb U_k)$ is strictly Hurewicz bounded.
\end{corollary}
\begin{proof} Since every compact set in $X$ is finite, we
have $(C(X,Y),\mathbb U_k) = (C(X,Y),\mathbb U_p)$.  Apply now the
previous theorem. \, \, $\Box$
 \end{proof}

\medskip
For the proof of the following proposition we use some ideas from
\cite[Example 2.6]{hernandez}.

\begin{proposition} \label{prop3} Let $(Y,d)$ be a non-bounded  $\sigma$-totally bounded metric
space. Then for every $n \in \mathbb N$, $Y^n$ is strictly {\sf
H}-bounded. However $Y^{\mathbb N}$ with the product uniformity is
not {\sf M}-bounded.
\end{proposition}
\begin{proof} Since $Y^n$, $n \in\mathbb N$, is $\sigma$-totally bounded, it is
strictly {\sf H}-bounded. We prove that $Y^{\mathbb N}$ with the
product uniformity $\mathbb U = \prod_{n\in\mathbb N} \mathbb V_n$
is not {\sf M}-bounded. (Here $\mathbb V_n$ is the (metric)
uniformity on $Y_n = Y$.) Let $\pi_n: Y^{\mathbb N} \to Y_n= Y$ be
the projection on the $n$-th coordinate space $Y= Y_n$. Let
$(U_n:n\in \mathbb N)$ be a sequence in $\mathbb U$; one can assume
that all $U_n$s are from the standard base for $\mathbb U$:
 $U_n = U_n(V_{n,1}, \dots, V_{n,n}) = \bigcap_{m\le n}(\pi_m\times
\pi_m)^{\gets}(V_{n,m})$, where $V_{n,m} \in \mathbb V_m \setminus
\{Y \times Y\}$ for $m\le n$. Let $(A_n:n\in\naturals)$ be an
arbitrary sequence of finite subsets of $Y^{\mathbb N}$. Since $Y$
is non-bounded, for each $n\in\mathbb N$, there is a point $y_n \in
Y_n \setminus (\pi_n\times \pi_n)(U_n)[\pi_n(A_n)]$. Then the point
$y= (y_n:n\in \mathbb N) \in Y^{\mathbb N} \setminus
\bigcup_{n\in\mathbb N}U_n[A_n]$.\, \, $\Box$
\end{proof}

\begin{theorem} \label{cp-m-bound} Let $X$ be a Tychonoff space and $(Y,d)$ be a non-bounded
hemicompact, arcwise connected  metric space. The following are
equivalent:
\begin{itemize}
\item[{\rm (1)}] $(C(X,Y),\mathbb U_p)$ is  strictly Hurewicz bounded;

\item[{\rm (2)}] $(C(X,Y),\mathbb U_p)$ is  strictly Menger bounded;

\item[{\rm (3)}] $(C(X,Y),\mathbb U_p)$ is  Hurewicz bounded;

\item[{\rm (4)}] $(C(X,Y),\mathbb U_p)$ is  Menger bounded;

\item[{\rm (5)}] $X$ is pseudocompact.
\end{itemize}
\end{theorem}
\begin{proof} Only $(4) \Rightarrow (5)$ need a proof. Suppose that $X$ is
not pseudocompact. There is a sequence $\{O_n: n \in \mathbb N\}$ of
open sets such that the family $\{\overline{O_n}: n \in \mathbb N\}$
is discrete. Choose for every $n \in \mathbb N$, $x_n \in O_n$ and
put
\[
H = \{x_n: n \in \mathbb N\}.
\]
It is easy to verify that every function $f: H \to Y$ can be
continuously extended to a continuous function $f^*: X \to Y$.
(There is $y \in Y$ such that $y \ne f(x_n)$ for every $n \in
\mathbb N$. Let $n \in \mathbb N$ and let $C_n$ be an arc containing
$f(x_n)$ and $y$. There is a continuous function $f_n:
\overline{O_n} \to C_n$ such that
\[
f_n(x) = \left\{
                \begin{array}{ll} y, & \hbox{for every $x \in \overline{O_n} \setminus
O_n$,}\\
f(x_n),  & \hbox{for $x = x_n$.}
\end{array}
\right.
\]
Define the function $f^*: X \to Y$ as follows:
\[
f^*(x) = \left\{
                \begin{array}{ll} f_n(x), & \hbox{if $x \in \overline{O_n}$, $n \in \mathbb
                N$,}\\
                y, & \hbox{otherwise.}
                \end{array}
                \right.
\]
Of course, $f^{*}$ is continuous. The mapping $\pi: (C(X,Y),\mathbb
U_p) \to (C(H,Y),\mathbb U_p)$ defined by $\pi(f) = f\restriction H$
is uniformly continuous and onto. Thus $(C(H,Y),\mathbb U_p)$ is
Menger bounded. However $C(H,Y) = Y^{\mathbb N}$, and $Y^{\mathbb
N}$ with the product uniformity is not Menger bounded by Proposition
\ref{prop3}. \, \, $\Box$
\end{proof}

\begin{theorem} Let $X$ be a Tychonoff space and $(Y,d)$ be a separable metric
space. If every compact set in $X$ is metrizable, then
$(C(X,Y),\mathbb U_k)$ is $\omega$-bounded;
\end{theorem}
\begin{proof} Let $W(K,\varepsilon) \in \mathbb U_k$, where $K$ is a compact subspace of $X$ and
$\varepsilon  >0$. We will find a countable family $\mathcal F
\subset C(X,Y)$ such that
\[
C(X,Y) \subset W(K,\varepsilon)[\mathcal F].
\]
Consider the set $F = \{f \restriction K: f \in C(X,Y)\} \subset
C(K,Y)$. The separability and metrizability of $(C(K,Y),\tau_u)$
implies that $F$ is a separable subspace of $(C(K,Y),\tau_u)$. Let
$\{f_i \restriction K: i \in \mathbb N\}$ be a countable dense
subset of $F$, and let
\[
\mathcal F= \{f_i: i \in \mathbb N\}.
\]
Then one can easily verify that $C(X,Y) \subset
W(K,\varepsilon)[\mathcal F]$. \, \, $\Box$
\end{proof}

\medskip
\noindent {\bf Note.}  If $Y = \mathbb R^n$, $n \in \mathbb N$, then
$\omega$-boundedness of $(C(X,Y),\mathbb U_k)$ implies that every
compact set in $X$ is metrizable (see \cite{hola-koc}).

\begin{theorem} Let $X$ be a hemicompact space, and $(Y,d)$ a metric
space. The following are equivalent:
\begin{itemize}
\item[{\rm (1)}] $(C(X,Y),\mathbb U_k)$ is strictly {\sf H}-bounded;

\item[{\rm (2)}]  $(C(X,Y),\mathbb U_k)$ is strictly {\sf M}-bounded;

\item[{\rm (3)}] $(C(X,Y),\mathbb U_k)$ is {\sf H}-bounded;

\item[{\rm (4)}]  $(C(X,Y),\mathbb U_k)$ is $\sigma$-totally bounded.
\end{itemize}
\end{theorem}
\begin{proof} It is known \cite{MN} that hemicompactness of $X$
implies that the space $(C(X,Y),\mathbb U_k)$ is metrizable. Then
apply Proposition \ref{count.base}. \, \, $\Box$
 \end{proof}

\medskip
Similarly, applying Proposition \ref{prop2} we have the following
result.

\begin{theorem} Let $X$ be a hemicompact space, and $(Y,d)$ a metric
space. The following are equivalent:
\begin{itemize}
\item[{\rm  (1)}] $\mathcal F \subset (C(X,Y),\mathbb U_k)$ is strictly {\sf R}-bounded;

\item[{\rm (2)}] $\mathcal F$ is countable.
\end{itemize}
\end{theorem}

\subsection{$(C(X,Y),\mathbb U_u)$ and $(C(X,Y),\mathbb U_m)$}

We are going now to investigate function spaces with uniformities
$\mathbb U_u$ and $\mathbb U_m$.

\begin{theorem} Let $X$ be a Tychonoff space, $(Y,d)$ be a metric space. The
following are equivalent:
\begin{itemize}
\item[{\rm (1)}]  $(C(X,Y),\mathbb U_m)$ is $\omega$-bounded;

\item[{\rm (2)}]  $(C(X,Y),\mathbb U_u)$ is $\omega$-bounded;

\item[{\rm (3)}]  $(C(X,Y),\tau_u)$ is separable;

\item[{\rm (4)}] $X$ is compact metrizable and $(Y,d)$ is separable.
\end{itemize}
\end{theorem}
\begin{proof} $(1) \Rightarrow (2)$ is trivial, and
$(2) \Leftrightarrow (3)$ follows from Fact 1 since $(C(X,Y),\mathbb
U_u)$ is metrizable. (3) $\Leftrightarrow$ (4) is well known fact,
while (4) $\Rightarrow$ (1) follows from the fact that compactness
of $X$ implies that the uniformities $\mathbb U_m$ and $\mathbb U_u$
coincide. \, \, $\Box$
\end{proof}

\medskip
The following theorems are consequences of Propositions
\ref{count.base} and \ref{prop2} and the fact that $(C(X,Y),\mathbb
U_u)$ is metrizable.

\begin{theorem} \label{thm-shb} Let $X$ be a Tychonoff space and $(Y,d)$ a metric space. The
following are equivalent:
\begin{itemize}
\item[{\rm (1)}]  $(C(X,Y),\mathbb U_u)$ is strictly Hurewicz bounded;

\item[{\rm (2)}]  $(C(X,Y),\mathbb U_u)$ is strictly Menger bounded;

\item[{\rm (3)}]  $(C(X,Y),\mathbb U_u)$ is  Hurewicz bounded;

\item[{\rm (4)}] $(C(X,Y),\mathbb U_u)$ is $\sigma$-totally bounded.
\end{itemize}
\end{theorem}

\begin{theorem} Let $X$ be a Tychonoff space and $(Y,d)$ a metric
space. For $\mathcal F \subset (C(X,Y ), \mathbb  U_u)$ the
following are equivalent:
\begin{itemize}
\item[{\rm (1)}] $(\mathcal F, \mathbb U_u\restriction \mathcal F)$ is strictly {\sf R}-bounded;

\item[{\rm (2)}]  $\mathcal F$ is countable.
\end{itemize}
\end{theorem}

\begin{theorem} Let $X$ be a Tychonoff, and $(Y, d)$ be a metric space. The
following are equivalent:
\begin{itemize}
\item[{\rm (1)}] $(C(X, Y ), \mathbb U_u)$ is strictly {\sf R}-bounded;

\item[{\rm (2)}] $X$ is compact metrizable, $Y$ is countable and $C(X,Y)$ is
countable.
\end{itemize}
\end{theorem}

We have the following result.

\begin{theorem} \label{Y-Frechet} Let $X$ be a Tychonoff space and $Y$ be an infinite
closed convex subset of a Fr\'echet (= locally convex completely
metrizable) space $(Z,d)$ with a translation-invariant metric $d$ on
$Z$. The following are equivalent:
\begin{itemize}
\item[{\rm (1)}] $(C(X, Y),\tau_m)$ is $\sigma$-compact;

\item[{\rm (2)}] $(C(X,Y),\mathbb U_m)$ is  {\sf H}-bounded;

\item[{\rm (3)}] $(C(X, Y), \mathbb U_u)$ is {\sf H}-bounded;

\item[{\rm (4)}] $(C(X,Y), \tau_u)$ is $\sigma$-compact;

\item[{\rm (5)}] $X$ is finite and $Y$ is $\sigma$-compact;

\item[{\rm (6)}]  $(C(X,Y), \mathbb U_m)$ is strictly  {\sf M}-bounded;

\item[{\rm (7)}]  $(C(X, Y), \mathbb U_u)$ is strictly {\sf M}-bounded.
\end{itemize}
\end{theorem}
\begin{proof} (1) $\Rightarrow$ (2) and (2) $\Rightarrow$ (3) are trivial. Also (3) $\Rightarrow$
(4) follows from Theorem \ref{thm-shb}. [The  {\sf H}-boundedness of
$(C(X, Y), \mathbb U_u)$ implies its $\sigma$-total boundedness, and
the completeness of $(C(X, Y), \mathbb U_u)$ equipped with the sup
metric gives us the $\sigma$-compactness of $(C(X, Y),\tau_u)$.]

\smallskip
(4) $\Rightarrow$ (5) $X$ must be compact and metrizable since
$(C(X,Y),\tau_u)$ is separable. Suppose that $X$ is infinite and
$\rho$ is a compatible metric on $X$. Let $(Z,d)$ be a Fr\'echet
space such that $Y$ is an infinite closed convex subset of $(Z,d)$.

Let $\{K_n : n \in\mathbb N\}$ be a sequence of compact sets in
$(C(X, Y),\tau_u)$ such that
\[
C(X,Y) = \bigcup_{n\in\mathbb N}K_n.
\]
There must exist $n \in\mathbb N$ such that $K_n$ has a nonempty
interior. Let $f \in C(X,Y)$ and $\varepsilon > 0$ be such that
\[
H = \{g \in C(X,Y): d(f(x), g(x)) < \varepsilon\} \subset K_n.
\]
There must exist a sequence $(x_n)_{n \in \mathbb N}$ of different
points in $X$ converging to a point $x$. Clearly,
$(f(x_n))_{n\in\mathbb N}$ converges to $f(x)$. Let $V$ be an open
convex neighbourhood of the origin 0 of $Z$ such that $V \subset \{z
\in Z: d(z, 0) < \varepsilon/4\}$. Put
\[
L = (f(x) + V ) \cap Y.
\]
Then, of course, $L$ is also a convex set.

Choose $y \in L$, $y\neq f(x)$,  and put $\alpha = d(y, f(x))$.
Without loss of generality we can suppose that
\[
d(f(x), f(x_n)) < \alpha/2 \mbox{ for every } n \in\mathbb N.
\]
There is a sequence $(\eta_n)_{n\in\mathbb N}$ of positive reals
such that
\[
B(x_n, \eta_n) \cap B(x_m, \eta_m) = \emptyset \mbox{ for every } m
\neq n,
\]
and such that
\[
f(B(x_n, \eta_n)) \subset L \mbox{ for every } n \in\mathbb N.
\]
Let $n \in\mathbb N$. Put
\[
C_n = (B(x_n, \eta_n) \setminus S(x_n, \eta_n)) \cup \{x_n\}.
\]
Define the function $f_n : C_n \to L$ as follows:
\[
f_n(z) = \left\{ \begin{array}{ll} y, & \hbox{if $z=x_n$,}\\
f(z), & \hbox{otherwise.}
\end{array}\right.
\]
By Dugundji's extension theorem \cite{dugundji}, there is a
continuous extension $f_n^{\ast}: B(x_n, \eta_n) \to L$ of $f$. It
is easy to verify that the function $g_n : X \to Y$ defined by
\[
g_n(x) =\left\{
                 \begin{array}{ll}
                   f_n^{\ast}(x), & \hbox{for $x \in B(x_n, \eta_n)$;} \\
                   f(x), & \hbox{otherwise.}
                 \end{array}
               \right.
\]
is a continuous function from $X$ to $Y$. Realize that for $n \neq
m$ we have
\[
\sup\{d(g_n(x), g_m(x)) : x \in X\} \ge \alpha/2.
\]
Of course, $\{g_n: n\in \mathbb N\} \subset  K_n$, a contradiction,
since the sequence $\{g_n : n \in\mathbb N\}$ has no cluster point
in $(C(X, Y),\tau_u)$. Thus $X$ must be finite. Then $(C(X, Y),
\tau_u) = (C(X, Y),\tau_p)$ is $\sigma$-compact. Thus also $Y$ must
be $\sigma$-compact.

\smallskip
(5) $\Rightarrow$ (1) is clear. Further, (5) $\Rightarrow$ (6) is
evident, (6) $\Rightarrow$ (7) is clear, and (7) $\Rightarrow$ (3)
by Proposition \ref{count.base}. \, \, $\Box$
\end{proof}

\section*{Acknowledgements}

The first author would like to thank the support of Vega 2/0006/16.

\footnotesize{

}

\end{document}